\documentclass[11pt,oneside]{amsart}

\usepackage[english]{babel}    
\usepackage[utf8]{inputenc}    
\usepackage[T1]{fontenc}       
\usepackage{amssymb} 		    
\usepackage{mathtools}         
\usepackage{icomma}            
\usepackage{units}             
\usepackage{enumerate}         
\usepackage{array}             
\usepackage[usenames]{color}   
\usepackage{graphicx}          

\usepackage{dsfont}

\usepackage[scale=.83]{geometry}

\usepackage[bookmarks, colorlinks=true, linkcolor=black, anchorcolor=black,
citecolor=black, filecolor=black, menucolor=black, runcolor=black,
urlcolor=black, pdfencoding=unicode]{hyperref}  

\newtheorem{thm}{Theorem}
\newtheorem{lemma}[thm]{Lemma}
\newtheorem{prop}[thm]{Propositon}
\newtheorem{cor}[thm]{Corollary}
\theoremstyle{remark}
\newtheorem{remark}[thm]{Remark}

\newcommand{\RR}{\ensuremath{\mathbb{R}}}

\newcommand{\N}{\ensuremath{\mathbb{N}}}

\newcommand{\B}{\ensuremath{\mathbb{B}}}

\renewcommand{\P}{\ensuremath{\mathbb{P}}}
\newcommand{\E}{\ensuremath{\mathbb{E}}}
\newcommand{\F}{\mathcal{F}}
\newcommand{\X}{\mathcal{X}}
\newcommand{\Y}{\mathcal{Y}}
\newcommand{\absX}[1]{|\widetilde{X}_{#1}|}

\newcommand{\I}{\mathbf{I}}

\newenvironment{itemize*}{\vspace{-10pt}\begin{itemize}\setlength{\itemsep}{0pt}\setlength{\parskip}{2pt}}{\end{itemize}}
\newenvironment{enumerate*}{\vspace{-10pt}\begin{enumerate}\setlength{\itemsep}{0pt}\setlength{\parskip}{2pt}}{\end{enumerate}}
\newenvironment{description*}{\vspace{-12pt}\begin{description}\setlength{\itemsep}{0pt}\setlength{\parskip}{2pt}}{\end{description}}

\newcommand{\Sp}{\ensuremath{\mathbb{S}}}
\newcommand{\abs}[1]{\left|#1\right|}
\newcommand{\set}[2]{\left\{#1\ ; \ #2  \right\}}
\renewcommand{\d}{\mathrm{d}}
\newcommand{\ind}{\ensuremath{\mathds{1}}}

\newcommand{\indu}[2]{\ensuremath{\mathds{1}(U_{#1}#2)}}

\title{Projections of spherical Brownian motion}

\author{Aleksandar Mijatovi{\'c}}
\address{Department of Statistics, University of Warwick, \& The Alan Turing Institute, UK}
\email{a.mijatovic@warwick.ac.uk}
\author{Veno Mramor}
\address{Department of Statistics, University of Warwick, \& The Alan Turing Institute, UK}
\email{veno.mramor@gmail.com}
\author{Ger{\'o}nimo Uribe Bravo}
\address{Instituto de Matematicas, Universidad Nacional Aut{\'o}noma de M{\'e}xico, 
M{\'e}xico }
\email{geronimo@matem.unam.mx}

\thanks{AM is supported	by the EPSRC grant EP/P003818/1 and a Fellowship at The
Alan Turing Institute, sponsored by the Programme on Data-Centric Engineering
funded by Lloyd's Register Foundation; GUB supported by CoNaCyT grant FC-2016-1946 and UNAM-DGAPA-PAPIIT grant IN115217.}

\numberwithin{equation}{section}
\numberwithin{thm}{section}

\begin{document}
\begin{abstract}
We obtain a stochastic differential equation (SDE) satisfied by the first $n$
coordinates of a Brownian motion on the unit sphere in $\mathbb{R}^{n+\ell}$.
The SDE has non-Lipschitz coefficients but we are able to provide an analysis
of existence and pathwise uniqueness and show that they always hold. 
The square of the radial component is a Wright-Fisher diffusion with mutation and it features in a skew-product decomposition of the projected spherical Brownian motion. 
A more general SDE 
on the unit ball in  
$\mathbb{R}^{n+\ell}$
allows us to geometrically realize the Wright-Fisher diffusion with general non-negative parameters as the radial component of its solution.
\end{abstract}

\maketitle

\section{Introduction and main results}
The Theorem of Archimedes \cite{archimedes} states that the 
projection $\pi_1$
from the unit sphere
$\Sp^2\subset\RR^3$ to any coordinate (in $\RR^3$) preserves the 
uniform 
distribution; 
see \cite{MR3035124} and the references therein for a very modern account or \cite{MR3025901} for one with more probabilistic insight. 
In probabilistic language, 
if $U^{(2)}$ is a uniform random vector 
on $\Sp^2$,
then
$\pi_1(U^{(2)})$ is uniform on $[-1,1]$. 
In fact, this holds in any dimension $d\geq3$:
for a uniform random vector $U^{(d-1)}$ on the Euclidean unit sphere 
$\Sp^{d-1}:=\set{z\in \RR^{d}}{\abs{z}=1}$, its projection  
$\pi_{d-2}(U^{(d-1)})$ onto  any $d-2$ coordinates 
is uniform on the unit ball 
$\B^{d-2}:=\set{z \in \RR^{d-2}}{\abs{z}\le 1}$. 
A more general version of this result for spheres in the $p$-norm can be found in \cite{MR2123199}.

Since the 
uniform 
distribution 
on $\Sp^{d-1}$ is
the invariant measure for Brownian motion on the sphere,
it is natural to investigate the process obtained by projecting 
it to the ball 
$\B^{d-2}$.
Such a process ought to have 
a 
uniform distribution 
on $\B^{d-2}$
as its invariant measure. 
The aim of this paper is to give a complete characterization of such processes in terms of SDEs 
they satisfy and to deduce certain structural consequences of this characterisation. 

Let $Z$ be Brownian motion on the sphere $\Sp^{n+\ell-1}$ ($n,\ell\in \N$); 
although this is an instance of a Brownian motion on a Riemannian manifold, 
we will just use the Stroock representation (in It\^o form) and consider it as the solution of an equation 
\begin{equation} \label{BMsphere}
\d Z_t=(I-Z_tZ_t^\top) \d \widetilde{B}_t-\frac{n+\ell-1}{2}Z_t \d t,\qquad Z_0\in \Sp^{n+\ell-1},
\end{equation}
where 
$\widetilde B$ a Brownian motion on $\mathbb{R}^{n+\ell}$ and $I$ denotes the identity matrix of appropriate dimension 
(cf. \cite[Ch.3\S3, p. 83]{hsumanifolds}). 
\begin{prop}\label{projekcija}
Let $X$ denote the first $n$ coordinates of $Z$. 
Then, there exists a Brownian motion $B$ on $\RR^n$ such that
the  pair $(X,B)$ satisfies 
the SDE
\begin{align*}
\d X_t=\sigma(X_t)\d B_t  -\frac{n+\ell-1}{2}X_t \d t,\qquad X_0\in\B^n,
\intertext{where the volatility matrix $\sigma(x)$ takes the form}
\sigma(x)=I-\left(1-\sqrt{1-\abs{x}^2}\right)\frac{xx^\top}{\abs{x}^2}\ind(\abs{x}>0), \qquad x\in\B^n. 
\end{align*}
Pathwise uniqueness holds for this SDE and $X$ is a strong Markov process 
with a unique invariant measure which admits the  density 
$$h(x)=\frac{\Gamma((n+\ell)/2)}{\pi^{n/2}\Gamma(\ell/2)}\left(1-\abs{x}^2\right)^{(\ell-2)/2}\mathds{1}(\abs{x}\le1).$$
Furthermore, $U=|X|^2$ is a Wright-Fisher diffusion, i.e. 
there exists a scalar Brownian motion $\beta$ such that 
the pair $(U,\beta)$ 
satisfies the SDE
\begin{equation*}
\d U_t=2\sqrt{U_t(1-U_t)}d\, \beta_t+[ n(1-U_t)-\ell U_t ]\, dt. 
\end{equation*}
\end{prop}

\begin{remark} In the case $\ell=2$, the invariant measure of $X$
in Proposition~\ref{projekcija} is uniform on the unit ball $\B^n$, as expected from the theorem of Archimedes.
\end{remark}
\begin{remark}
The process $X$ enjoys a skew-product decomposition analogous to the one
of Brownian motion in $\RR^n$; it is a particular case of Theorem \ref{skew} below. 
\end{remark}
\begin{remark}
Since 
$\sigma(x)$
is the unique non-negative definite square root of the matrix $I-xx^\top$ for $x\in\B^n$, the SDE for the projected process $X$ 
implies that its infinitesimal generator equals 
$$\frac{1}{2}\sum_{i,j=1}^n(\delta_{ij}-x^ix^j)\frac{\partial^2}{\partial x^i
\partial x^j} -\sum_{i=1}^n\frac{n-1+\ell}{2}\frac{\partial}{\partial x^i}.$$
This operator equals the generator $\frac{1}{2}\Delta_{n-1+\ell,n}$ of a process
considered in~\cite{bakry1996remarques}. Our results show that the martingale
problem arising from this generator is well-posed.  
\end{remark}

More generally, and in order to state the skew-product decomposition, we will
consider the SDE
on the unit ball $\B^n$ given by 
\begin{equation} \label{enacbax}
\d X_t=\gamma(\abs{X_t})\sigma(X_t)\d B_t  -g(\abs{X_t})X_t \d t, 
\end{equation} with starting point $X_0=x_0\in\B^n$. 
Assume that $\gamma\colon [0,1]\to
(0,\infty)$ and  $g\colon[0,1]\to\RR$ are Lipschitz continuous\footnote{Note that by Lipschitz continuity
of the euclidean norm $\abs{\cdot}$, functions $\gamma(\abs{\cdot})$ and
$g(\abs{\cdot})$ are also Lipschitz.} and satisfy $\frac{g(1)}{\gamma^2(1)}\ge\frac{n-1}{2}.$
In particular, the above SDE extends the projected process $X$ of Proposition \ref{projekcija} to non-integer dimensions. 
The coefficients of SDE~\eqref{enacbax} are bounded and continuous implying that weak
existence holds  (see e.g.~\cite[Ch.~IV, Thm~2.2]{ikedawatanabe}). 
The function $\sigma$ is locally Lipschitz
only on the interior of the ball $\B^n$, making it impossible to apply the classical theory 
for the uniqueness of solutions of SDEs (note that we do not exclude the cases when either $X_0\in\Sp^{n-1}$ or the boundary
sphere is reached in finite time). 
The condition $ \frac{g(1)}{\gamma^2(1)}\ge\frac{n-1}{2}$ turns out to be
necessary for a solution to stay in the unit ball. 
In fact, if $\frac{g(1)}{\gamma^2(1)}=\frac{n-1}{2}$ and $X_0\in \Sp^{n-1}$, the solution $X$ is a 
Brownian motion on $\Sp^{n-1}$ time-changed by 
$t\mapsto\gamma^2(1)t$.

The radial component $R:=\abs{X}$ of a solution $X$ of
SDE~\eqref{enacbax} and its square $U:=\abs{X}^2$ are 
the unique strong solutions of the respective SDEs in~\eqref{enacbaradij} and
\eqref{prav_u} below; the latter reduces to the SDEs of Wright-Fisher diffusion with mutation in the setting of Proposition \ref{projekcija}. 
In particular, both processes are strong Markov.  
After a time-change, a pathwise comparison of  $U$ with a Wright-Fisher diffusion
implies that for $n\ge2$ the process $X$ never hits $0$ (see Lemma~\ref{enacbau} 
in Section~\ref{calc} below).
This enables us to define a time-change process
$S_s(t):=\int_s^t\frac{\gamma^2(R_u)}{R_u^2}\d u$, satisfying  
$\lim\limits_{t\to \infty}S_s(t)=\infty$, and its inverse $T_s\colon [0,\infty) \to
[s,\infty)$ (see Lemma~\ref{timechg} below).  
Moreover,
it turns out that 
$X$ possesses a skew-product decomposition analogous to the one
of Brownian motion on $\RR^n$.

\begin{thm}[Skew-product decomposition]\label{skew}
Let $n\ge2$ and $X$ be a solution of SDE~\eqref{enacbax}. 
Pick $s\in\RR_+:=[0,\infty)$
and assume
that either $s>0$ or $s=0$ and $X_0\neq 0$. Then the process
$\widehat{V}=(\widehat{V}_t)_{t\in\RR_+}$, given by $\widehat{V}_t:=X_{T_s(t)}/R_{T_s(t)}$, is a Brownian motion on
$\Sp^{n-1}$ (started at $\widehat{V}_0=X_s/R_s$) independent of $R$. 
Hence we obtain the skew-product decomposition
$X_t=R_t\widehat{V}_{S_s(t)}$ for $t\ge s$. Furthermore,
if $X_0=0$, then $\widehat{V}_t$ is uniformly distributed on $\Sp^{n-1}$ for
any $t>0$ and subsequently evolves as a stationary Brownian motion on the sphere. 
\end{thm}

Since the skew-product decomposition expresses $X$ as a measurable functional of a pair of independent processes 
$(R,\widehat{V})$ with given distributions, the following result holds. 
\begin{cor}\label{uniqlaw}
	Uniqueness in law holds for SDE~\eqref{enacbax}.
\end{cor}

The pathwise uniqueness of SDE~\eqref{enacbax} is more delicate because 
$\sigma$ is not Lipschitz at the boundary of the ball $\B^n$. 
Since $\sigma$ is locally Lipschitz, pathwise uniqueness holds up to the first hitting time 
of the boundary by well established arguments.  
Hence, if 
$\frac{g(u)}{\gamma^2(u)}\ge\frac{n-1}{2}+1$ holds for $u$ sufficiently close to one, $\abs{X_0}<1$, and
$X$ never visits the boundary of $\B^n$, then pathwise uniqueness holds. 
If $\frac{g(1)}{\gamma^2(1)}=\frac{n-1}{2}$, $X$ behaves as time-changed Brownian motion 
on $\Sp^{n-1}$ after the first time it hits the boundary and hence 
pathwise uniqueness also holds in this case. 
If $n=1$, the SDE~\eqref{enacbax} simplifies to
$\d X^1_{t}=
\gamma(\abs{X^1_t})\sqrt{1-(X^1_{t})^2}\d{B}_t-{g}(\abs{X^1_t})X^1_t\d t$
and pathwise uniqueness holds by a theorem of Yamada and Watanabe \cite[IX.3.5]{revuzyor}. 
Pathwise uniqueness of the similar looking equation 
$\d X_t=\gamma(\abs{X_t})(1-\abs{X_t}^2)^{1/2}\d B_t  -g(\abs{X_t})X_t \d t $ 
on $\B^n$, where $\gamma$ and $g$ are positive Lipschitz function and $\frac{g(1)}{\gamma^2(1)}$
is sufficiently large,
was established by DeBlassie~\cite{deblassie-uniq} by a clever generalisation 
of the idea in~\cite{swart-uniq}.
Note that the diffusion coefficient  in the SDE~\eqref{enacbax}
depends on $x$ and not just on its length $\abs{x}$, making it impossible to 
apply the result of~\cite{deblassie-uniq}.
However, it is possible to adapt the method of~\cite{deblassie-uniq} to our setting
and obtain:

\begin{thm}\label{pathuniq}
If $\frac{g(1)}{\gamma^2(1)}-\frac{n-1}{2}>\sqrt{2}-1\doteq0.4142$, then pathwise uniqueness holds for the SDE~\eqref{enacbax}.
\end{thm}
The remaining cases, when $n\ge2$ and
$\frac{g(1)}{\gamma^2(1)}-\frac{n-1}{2}\in(0,\sqrt{2}-1]$,  
are left open.

\section{
Characterization of the projected process
} \label{calc}
     We are interested in the process consisting of the first  $n$ coordinates of Brownian motion on the sphere $\Sp^{n-1+\ell}.$  One way of constructing such a Brownian motion is via the Stroock representation, i.e. a solution to the SDE \eqref{BMsphere}.
     Note that coefficients of SDE~\eqref{BMsphere} are locally Lipschitz continuous, so pathwise uniqueness holds. 

	In the following proof as well as in the rest of the paper superscripts will denote components of vectors, e.g. $X^i$ means $i$-th component of process $X$. When we wish to express powers we will enclose variables in additional sets of parentheses.
    
     \begin{proof}[Proof of Proposition \ref{projekcija}]
	 Let $X'$ denote last $\ell$ coordinates of $Z$ and similarly split $\widetilde{B}=(\widetilde{B}^1,\widetilde{B}^2).$ We claim that the process $B$ given by the equation    	
     		\begin{align*}
     	B_t&=\int_0^t\sigma(X_s)\d \widetilde{B}^1_s+\int_0^t\left(-X_s{X'_s}^\top(1-\abs{X_s}^2)^{-1/2}\ind(\abs{X_s}<1)+X_sz^\top\ind(\abs{X_s}=1)\right)\d \widetilde{B}_t^2,
     	\end{align*}
     		 where $z\in \Sp^{\ell-1}$ is an arbitrary (fixed) unit vector, is an $n$\babelhyphen{nobreak}dimensional Brownian motion.
     	
 Note that $\abs{{X'_t}}^2=\abs{Z_t}^2-\abs{X_t}^2=1-\abs{X_t}^2$. Furthermore, let us consider the $n\times (n+\ell)$ matrix $$A_t:=\begin{bmatrix}\sigma(X_t), -X_t{X'_t}^\top(1-\abs{X_t}^2)^{-1/2}\ind(\abs{X_t}<1)+X_tz^\top\ind(\abs{X_t}=1) \end{bmatrix}=:\begin{bmatrix}\sigma(X_t),D_t \end{bmatrix}.$$ 
 The choice of the constant vector $z$ in the definition of $A_t$ will turn out not to be relevant as we will see that the time spent by $X$ at the boundary of the ball has Lebesgue measure zero. We can compute $\sigma(X_t)^2=\sigma(X_t)\sigma(X_t)^\top=I-X_tX_t^\top$ and $D_tD_t^\top=X_tX_t^\top,$ 
 so  it follows that $ A_tA_t^\top=\sigma(X_t)\sigma(X_t)^\top+D_tD_t^\top=I.$
     Since  $B$ is defined by $B_t=\int_0^tA_s\d \widetilde{B}_s$ and it is a continuous local martingale with quadratic variation $\langle {B}^i,{B}^j\rangle_t=\int_0^t(A_sA_s^\top)_{ij}\d s=\delta_{ij}t,$ it is $n$\babelhyphen{nobreak}dimensional Brownian motion by Levy's characterization theorem. 
           Further calculations show that 
          $\sigma(X_t)D_t=-X_t{X'_t}^\top$ and finally, the facts and the definition of $Z$, imply the SDE satisfied by $X$: 
    \begin{align*}
        \d X_t&=(I-X_tX_t^\top)\d \widetilde{B}_t^1-X_t{X'_t}^\top\d \widetilde{B}_t^2-\frac{n-1+\ell}{2}X_t\d t\\
        &=\sigma(X_t)^2\d \widetilde{B}_t^1+\sigma(X_t)D_t\d \widetilde{B}_t^2-\frac{n-1+\ell}{2}X_t\d t\\
        &=\sigma(X_t)A_t\d \widetilde{B}_t-\frac{n-1+\ell}{2}X_t\d t=\sigma(X_t)\d B_t-\frac{n-1+\ell}{2}X_t\d t. 
    \end{align*}
    The above SDE is just a special case of SDE \eqref{enacbax} with $\gamma \equiv 1$ and $g\equiv \frac{n-1+\ell}{2}$, therefore pathwise uniqueness holds immediately by Theorem~\ref{pathuniq} since $\frac{g(1)}{\gamma^2(1)}-\frac{n-1}{2}=\frac{\ell}{2}>\sqrt{2}-1$ for $\ell \in \N$. Consequently $X$ is a strong Markov process. Furthermore, Lemma~\ref{enacbau} shows that $U=\abs{X}^2$ is Wright-Fisher diffusion with  mutation rates $n$ and $\ell$. When  $n\ge2$ this also helps us find invariant measure for process $X$ since we can use skew-product decomposition in Theorem~\ref{skew}. The invariant measure for Wright-Fisher diffusion $U$ is given by
    $\mathrm{Beta}(n/2,\ell/2)$ distribution. Hence 
    $g(r)=B_{n,\ell}r^{n-1}(1-r^2)^{(\ell-2)/2}\mathds{1}(r\in [0,1])$ 
    is the density of the 
    invariant measure of $R$.
    The invariant  measure for Brownian motion on a sphere is a normalised uniform measure. This
    continues to hold for the time-changed Brownian motion on a sphere as long as 
    the time change is independent of Brownian motion. So let us suppose that the initial distribution of the
    process $X$ has the density $h$ from Proposition~\ref{projekcija}. Then, using polar coordinates and the skew-product
    decomposition, the density of $R_0$ is 
    $g$. 
    Since this density is invariant for $R$, $R_t$ has density $g$ 
    for all $t\ge0$. The time changed Brownian motion on a sphere also
    remains uniformly distributed and reversing polar coordinates we
    get that $X_t$ has density $h$ for any $t$. 
        
    In the case $n=1$ we see that the process $X^1$ satisfies the SDE $\d X^1_{t}=
    \sqrt{1-(X^1_{t})^2}\d{B}_t-\frac{\ell}{2}X^1_t\d t$, and by the forward Kolmogorov equation 
    the invariant density can easily be seen to be equal to
    \begin{equation*}
    	h(x)=\frac{\Gamma((1+\ell)/2)}{\pi^{1/2}\Gamma(\ell/2)}\left(1-x^2\right)^{(\ell-2)/2}\mathds{1}(\abs{x}\le1). \qedhere
    \end{equation*}
    \end{proof}

\subsection{Skew-product decomposition for SDE~\eqref{enacbax}}
     The solution of \eqref{enacbax} naturally lives on the closed unit ball
$\B^n$. To better understand such a process it is crucial to understand its
radial component (or its square) and in particular if and when it hits the boundary
(or zero). Since the square of the radial component will be shown to be closely related to Wright-Fisher diffusion let us first collect some known fact about the latter. Fix non-negative parameters $\alpha$ and $\beta$. 
The SDE $\d
\mathcal{Z}_t=2\sqrt{\mathcal{Z}_t(1-\mathcal{Z}_t)}\d B_t +
(\alpha(1-\mathcal{Z}_t)-\beta \mathcal{Z}_t)\d t$ has a unique strong solution which is called the 
Wright-Fisher diffusion;  denote it by $\mathrm{WF}(\alpha,\beta)$. It is possible to define
$\mathrm{WF}(\alpha,\beta)$  for negative $\alpha$ (resp. $\beta$), but
then the solution has a finite lifetime equal to the first hitting time of $0$
(resp. $1$).   Pathwise uniqueness is a consequence of $1/2$-H\"older continuity of diffusion coefficient and then applying theorem of Yamada and Watanabe \cite[IX.3.5]{revuzyor}.  Furthermore, it is well-known\footnote{It can for example be seen from Lemma~\ref{kvocient} and well-known facts about hitting of $0$ of Bessel processes.} that $\mathrm{WF}(\alpha,\beta)$  never hits $0$  for $\alpha\ge2$ and never hits $1$ for $\beta\ge2$.

Some of the facts about square of radial component of solution to SDE \eqref{enacbax} are summarized in the following lemma.
     
     \begin{lemma}\label{enacbau}
        Let $X$ be a solution of \eqref{enacbax} where $\frac{g(1)}{\gamma^2(1)}\ge\frac{n-1}{2}$. 
        Then the process $U=\abs{X}^2$ satisfies the SDE
        \begin{equation} \label{prav_u}
        	\d U_t=2\widetilde{\gamma}(U_t)\sqrt{U_t(1-U_t)}\d\theta_t + \widetilde{\gamma}^2(U_t)\left(n(1-U_t)-\left(\frac{2\widetilde{g}(U_t)}{\widetilde{\gamma}^2(U_t)}-(n-1)\right)U_t\right)\d t,
        \end{equation}
        where $\widetilde{g}(u):=g(\sqrt{u}), \widetilde{\gamma}(u):=\gamma(\sqrt{u})$, and the
       $\F_t$-Brownian motion $\theta$ is defined by $$ \theta_t=\sum\limits_{i=1}^{n}\int_0^t\frac{X^i_{s}}{\sqrt{U_s}}\indu{s}{>0}\d {B}_s^i+\int_0^t\indu{s}{=0}\d \chi_s,$$ where the scalar Brownian motion $\chi$ is independent of ${B}$.

        For $n\ge2$, the process $U$ never hits $0$ and if $\frac{g(u)}{\gamma^2(u)}\ge\frac{n-1}{2}+1$ holds near $1$, 
        then $U$ never hits $1$.
     \end{lemma}
     \begin{proof}
		Since $U=\abs{X}^2$ and $ \langle X^{i}\rangle_t=\gamma^2(\abs{X_t})(1-(X^i_{t})^2)\d t$ we can apply It\^o's formula and get
     \begin{align*}
        \d U_t
        &=2\gamma(\abs{X_t})\sum_{j=1}^{n}X^j_{t}\sqrt{1-U_t}\d {B}^j_t -2g(\abs{X_t})\sum_{i=1}^{n}(X^i_{t})^2\d t+\gamma^2(\abs{X_t})\sum_{i=1}^{n}(1-(X^i_{t})^2)\d t
    \end{align*} which in turn yields \eqref{prav_u}. Moreover, since $\langle\theta\rangle_t=t$, the continuous local martingale $\theta$ is a Brownian motion by Levy's characterization theorem. 

       We can slightly simplify the equation \eqref{prav_u} by time-change without affecting the boundary hitting properties. Define $q$ by
       \begin{equation*}
       q_t:=\int_0^{t}\gamma^2(\abs{X_u})\d u\text{ and its  inverse }\widetilde{q}_t:=\inf\set{u\ge 0}{q(u)=t}.
       \end{equation*} Then $\widehat{U}_t:=U_{\widetilde{q}_t}$ satisfies the SDE\begin{equation} \label{tchgU}
 	\d \widehat{U}_t=2\sqrt{\widehat{U}_t(1-\widehat{U}_t)}\d\widetilde{\theta}_t + \left( n(1-\widehat{U}_t)-\left(\frac{2\widetilde{g}(\widehat{U}_t)}{\widetilde{\gamma}^2(\widehat{U}_t)}-(n-1)\right)\widehat{U}_t\right)\d t
 \end{equation}
  where $\widetilde{\theta}_t=\int_0^{\widetilde{q}_t}\gamma(\abs{X_u})\d
\theta_u$ is also a Brownian motion. This is almost the same equation as that of a Wright-Fisher diffusion. The volatility term is exactly the same so that the difference appears only in the drift term, which is, nevertheless, still 
Lipschitz
continuous\footnote{The function $u\mapsto
uf(\sqrt{u})$ is Lipschitz if $f$ is. Hence, $u\mapsto u\frac{\widetilde{g}(u)}{\widetilde{\gamma}^2(u)}$ is Lipschitz continuous. 
}. Again we can use the  Yamada-Watanabe Theorem \cite[Theorem~3.5 in Ch. IX]{revuzyor} and conclude that
pathwise uniqueness holds for
the  SDE~\eqref{tchgU}. 
   Since $   \frac{g(1)}{\gamma^2(1)}\ge\frac{n-1}{2}$ we can consider SDE~\eqref{tchgU} with the increased drift  
   \begin{equation*}
   n(1-u)-\left(\frac{2\widetilde{g}(u)}{\widetilde{\gamma}^2(u)}- \frac{2\widetilde{g}(1)}{\widetilde{\gamma}^2(1)}\right)u
   \end{equation*}and for $\widehat{U}_0=1$ such an equation has the unique solution $\widehat{U}_t\equiv1$. By using the comparison theorem for SDEs, as in  \cite[ Ch. IX, (3.7)]{revuzyor}, we deduce that the inequality $\widehat{U}_t\le1$ holds for the solution of \eqref{tchgU} as long as $\widehat{U}_0\le1$ a.s. By denoting
   \begin{equation*}
   M:=\max_{u\in [0,1]}\left(\frac{2\widetilde{g}(u)}{\widetilde{\gamma}^2(u)}-(n-1)\right)\ge0, 
   \end{equation*}another application of the comparison theorem shows that the solution of \eqref{tchgU} is always larger than $\mathrm{WF}(n,M)$ started at the same point, so it is always non-negative and therefore $U_t\in [0,1]$ for $U_0\in[0,1]$ a.s. which also means it has infinite lifetime.
    This second use of comparison theorem also shows, that for $n\ge2$, the processes  $\widehat{U}$ and $U$ never hit $0$ unless they start there. Similarly, if $\frac{g(u)}{\gamma^2(u)}\ge\frac{n-1}{2}+1$ holds near $1$, we could at least locally (near 1) use comparison theorem to show that $\widehat{U}$ is smaller than $\mathrm{WF}(n,2)$ started at the same point and therefore under such conditions $\widehat{U}$ (and $U$) never hits $1$.
     \end{proof}

	Before we can prove Theorem~\ref{skew} we need two additional lemmas. The first one expresses the Wright-Fisher diffusion as a certain skew-product of two independent squared Bessel processes. For notation and basic facts  about (squared) Bessel processes see \cite[Chapter XI]{revuzyor}. 
	\begin{lemma} \label{kvocient}
		Let $\alpha\ge0,\beta\in \RR$ and let $\X$ be $\mathrm{BES}Q^\alpha$ process started at $x_0\ge0$ and $\Y$ be independent $\mathrm{BES}Q^\beta$ process started at $y_0\ge0$ such that $x_0+y_0>0$. 
Let $T_0({\Y}):=\inf\set{t\ge0}{\Y_t=0}$. 
 Define the continuous additive functional $\rho$ as
 \begin{equation*}\rho_t=\int_0^t\frac{1}{\X_u+\Y_u}\d u\text{ and its inverse }\zeta_t=\inf\set{u\ge0}{\rho_u=t}. 
 \end{equation*}Let ${U_t}:=\frac{\X_t}{\X_t+\Y_t}$ for $t<T_0({\Y})$. Then $\widehat{U}_t:={U}_{\zeta_t}$ for $t<\rho_{T_0(\Y)}$ is a $\mathrm{WF}(\alpha,\beta)$ started at $\frac{x_0}{x_0+y_0}$ and $\widehat{U}$ is independent of $\X+\Y.$
	\end{lemma}
\begin{proof} 
The lemma is proved in \cite[Propositon 8]{WarrenYorJacobi} for
non-negative coefficients $\alpha,\beta$ and the larger stopping time
$T_0(\X+\Y)=\inf\set{t\ge0}{\X_t+\Y_t=0}$, where they use the term Jacobi
diffusion for particular Wright-Fisher diffusions. The same proof works for 
$\beta<0$, since 
$T_0({\Y})<T_0(\X+\Y)$
and hence on the stochastic interval $[0,T_0({\Y}))$ 
all the processes in the proof of~\cite[Propositon~8]{WarrenYorJacobi}
are well defined and all calculations stay exactly the same.  
\end{proof}
	
	Our next lemma summarizes facts about the time-change used in Theorem~\ref{skew}.
	\begin{lemma} \label{timechg} 
		Let $n\ge2$ and either $s>0$ or $s=0$ and $X_0\neq 0$. Define
		\begin{equation*}S_s(t):=\int_s^t\frac{\gamma^2(R_u)}{R_u^2}\d u . 
		\end{equation*}Then $S_s\colon [s,\infty)\to \RR_+$ is continuous, strictly increasing, and $\lim\limits_{t\to \infty}S_s(t)=\infty$. Its right continuous inverse  $T_s\colon \RR_+\to [s,\infty)$ 
		is also continuous and strictly increasing with $T_s(0)=s.$ Furthermore, if $X_0=0$, then $\lim\limits_{s\downarrow 0}S_s(t)=\infty$ holds for any $t>0$. 
	\end{lemma}
     \begin{proof}Since a.s. $R_t^2\in (0,1]$ for any $t\ge s$, $R_t$ is continuous in $t$, and also $\gamma(R_t)$ is continuous in $t$ and bounded away from $0$, everything but the last claim follows from classical results on inverses of strictly increasing continuous functions. Since $\lim\limits_{s\downarrow0}\int_s^t\frac{\gamma^2(R_u)}{R_u^2}\d u=\int_0^t\frac{\gamma^2(R_u)}{R_u^2}\d u$, in order to prove the last claim it is enough to prove that $\int_0^t\frac{\gamma^2(R_u)}{R_u^2}\d u=\infty$ for any $t>0$. Changing variables with the time change $q_t$ from Lemma~\ref{enacbau} we get
     \begin{equation*}\int_0^t\frac{\gamma^2(R_u)}{R_u^2}\d u=\int_0^t\widehat{U}_{q_u}^{-1}\d q_u=\int_0^{q_t}\widehat{U}^{-1}_u\d u. 
     \end{equation*}By the comparison theorem for SDEs and strict positivity of $q_t$ it is therefore enough to prove that $\int_0^{t\wedge T_1(\mathcal{U})}\mathcal{U}_u^{-1}\d u=\infty$ for any $t>0$, where $\mathcal{U}$ is Wright-Fisher($n,m$) diffusion started at $0$,  
     \begin{equation*}
     m:=\min_{u\in [0,1]}\left(\frac{2\widetilde{g}(u)}{\widetilde{\gamma}^2(u)}-(n-1)\right) 
     \end{equation*}is a possibly negative constant, and $T_1(\mathcal{U})=\inf\set{t\ge0}{\mathcal{U}_t=1}$ is strictly positive. By Lemma~\ref{kvocient} we can find a $\mathrm{BES}Q^n$ process $\X$ started at $0$ and a $\mathrm{BES}Q^{m}$ process $\Y$ not started at $0$ such that $\mathcal{U}_t=\frac{\X_{\zeta_t}}{\X_{\zeta_t}+\Y_{\zeta_t}}$ holds for $t<T_1(\mathcal{U})=\rho_{T_0(\Y)}$ with the time changes $\zeta$ and $\rho$ defined as in Lemma~\ref{kvocient}. Using the time change formula for the Lebesgue-Stieltjes integral we get
     \begin{equation*}
     \int_0^{t\wedge T_1(\mathcal{U})}\mathcal{U}_t^{-1}\d u=\int_0^{t\wedge T_1(\mathcal{U})}\frac{\X_{\zeta_u}+\Y_{\zeta_u}}{\X_{\zeta_u}}\d u=\int_0^{t\wedge T_1(\mathcal{U})}\X_{\zeta_u}^{-1}\d \zeta_u=\int_0^{\zeta_t\wedge \zeta_{T_1(\mathcal{U})}}\X_u^{-1}\d u. 
     \end{equation*}This corresponds to the integral $\int_0^{\zeta_t\wedge \zeta_{T_1(\mathcal{U})}}h(\sqrt{\X_u})\d u$ where $h(x)=x^{-2}$ and $\sqrt{\X}$ is Bessel process started at $0$ with parameter $n\ge2$. Since $\zeta_{T_1(\mathcal{U})}=T_0(\Y)>0$, $\zeta_t>0$ for any $t>0$ and $-2\le-(2\wedge n)$, Corollary 2.4 in \cite{chernybesselconv} implies that the integral $S_0(t)$ diverges. 
     \end{proof}

     \begin{proof}[Proof of Theorem~\ref{skew}]
        Let us define function $\kappa (x^1,\ldots,x^n):=\left(\sum_{i=1}^{n}(x^i)^2\right)^{1/2}.$ Then $R_t=\kappa(X_t), V^i_{t}=\frac{\partial \kappa}{\partial x^i}(X_t)$ and a straightforward computation shows that
        \begin{equation*}\frac{\partial^2 \kappa}{\partial x^i\partial x^j}(x)=r^{-3}(\delta_{ij}r^2-x^ix^j)
        \text{ and }
        \frac{\partial^3 \kappa}{\partial x^i\partial x^j\partial x^k}(x)=3r^{-5}x^ix^jx^k-r^{-3}\left(x^i\delta_{jk}+x^j\delta_{ik}+x^k\delta_{ij}\right),
        \end{equation*}where $r:=\kappa(x).$ Note that also $\d\langle X^{i}, X^{j}\rangle_t=\gamma^2(\abs{X_t})(\delta_{ij}-X^i_{t}X^j_{t})\d t$. Now we can use the It\^o formula to get equations for $R_t$ and $V_t.$ 
        First, 
        \begin{align*}
        \d R_t
        &=\gamma(\abs{X_t})\sum_{j=1}^{n}\frac{X^j_{t}}{R_t}\sqrt{1-R_t^2}\d {B}^j_t -g(\abs{X_t})R_t\d t + \frac{\gamma^2(\abs{X_t})}{2}\sum_{i,j=1}^{n}\frac{\delta_{ij}R_t^2-X^i_{t}X^j_{t}}{R_t^3}(\delta_{ij}-X^i_{t}X^j_{t})\d t\\
         &=\gamma(R_t)\sum_{j=1}^{n}V^j_{t}\sqrt{1-R_t^2}\d {B}^j_t + \frac{(n-1)\gamma^2(R_t) -2g(R_t) R_t^2}{2R_t}\d t.
    \end{align*}
  Since $\d\theta_t=\sum_{j=1}^{n}V^j_{t}\d {B}^j_t$ is also $1$\babelhyphen{nobreak}dimensional Brownian motion (actually, it is the same process as in Lemma~\ref{enacbau}) we write the above equation in a more compact way as 
  \begin{equation} \label{enacbaradij}
  	\d R_t=\gamma(R_t)\sqrt{1-R_t^2}\d\theta_t  + ((n-1)\gamma^2(R_t) -2g(R_t) R_t^2)/({2R_t})\d t.
  \end{equation}
   We can also write an equation for $V^i_{t}.$
      \begin{align*}
        \d V^i_{t}&=\sum_{j=1}^{n}\frac{\delta_{ij}R_t^2-X^i_{t}X^j_{t}}{R_t^3}\d X^j_{t} +\frac{1}{2}\sum_{j,k=1}^{n}\left(\dfrac{\partial^3 \kappa}{\partial x^i\partial x^j\partial x^k}(X_t)\right)\d \langle X^{j}, X^{k}\rangle_t
    \end{align*}
    Similar calculations as above show that first sum is equal to 
   $\gamma(R_t)\sum_{j=1}^{n}\frac{\delta_{ij}-V^i_{t}V^j_{t}}{R_t}\d {B}_t^j$
     where the drift part vanishes and we compute the second sum as
    \begin{align*}
        &\frac{\gamma^2(R_t)}{2}R_t^{-5}\sum_{j,k=1}^{n}\left(3X^i_{t}X^j_{t}X^k_{t} -(X^i_{t}\delta_{jk}+X^j_{t}\delta_{ik}+X^k_{t}\delta_{ij})R_t^2)(\delta_{jk}-X^j_{t}X^k_{t}\right)\d t=\\
        &=\frac{\gamma^2(R_t)}{2}R_t^{-5}\left(3X^i_{t}R_t^2-(nX^i_{t}+X^i_{t}+X^i_{t})R_t^2-3X^i_{t}R_t^4+(X^i_{t}+X^i_{t}+X^i_{t})R_t^4\right)\d t\\
        &=-\gamma(R_t)\frac{n-1}{2}\frac{X^i_{t}}{R_t^3}\d t=-\frac{\gamma(R_t)}{R_t^2}\frac{n-1}{2}V^i_{t}\d t. 
    \end{align*}Altogether, in vector form, we get $ \d V_t= \frac{\gamma(R_t)}{R_t}(I-V_{t}V^\top_{t})\d {B}_t -\frac{\gamma^2(R_t)}{R_t^2}\cdot\frac{n-1}{2}V_{t}\d t.$    
    To show that the process $\widehat{V}_t=V_{T_t}$ is Brownian motion on a sphere it is sufficient to show that it satisfies SDE~\eqref{BMsphere}.
    Change of time formula for the It\^o and the Lebesgue-Stieltjes integrals immediately shows that  $\widehat{V}_t-\widehat{V}_0=\int_0^t(I-\widehat{V}_u\widehat{V}_u^\top)\d \widetilde{W}_u-\int_0^t\frac{n-1}{2}\widehat{V}_u\d u$, where $\widetilde{W}_t =  \int_s^{T_s(t)} \frac{\gamma(R_u)}{R_u}\d {B}_u$ is Brownian motion. This immediately implies that $\widehat{V}$ is Brownian motion on $\Sp^{n-1}$ but does not allow us to conclude that $R$ and $\widehat{V}$ are independent. With this in mind we modify the Brownian motion driving the SDE for $\widehat{V}$. Let us enlarge the probability space to accommodate another scalar Brownian motion $\xi$ which is independent of $B$ and define a continuous local martingale
    $ W_t =  \int_s^{T_s(t)} \frac{\gamma(R_u)}{R_u}(I-V_{u}V^\top_{u})\d {B}_u +\int_s^{T_s(t)}\frac{\gamma(R_u)}{R_u}V_{u} \d \xi_u.$ Then we can compute $\langle W^i,W^j \rangle_t=\int_s^{T_s(t)}\delta_{ij}\frac{\gamma^2(R_u)}{R_u^2}\d u=\delta_{ij}t,$ so $W$ is $\mathcal{G}_t$-Brownian motion, where $\mathcal{G}_t:=\F_{T_s(t)	}$. Since
    $(I-V_uV_u^\top)\cdot \begin{bmatrix}\frac{\gamma(R_u)}{R_u}(I-V_{u}V^\top_{u}) ,& \frac{\gamma(R_u)}{R_u}V_{u} \end{bmatrix}=\begin{bmatrix}\frac{\gamma(R_u)}{R_u}(I-V_{u}V^\top_{u}),& 0 \end{bmatrix} $
    we can use change of time formula for stochastic and Lebesgue-Stieltjes integral \cite[Chapter V, \S 1]{revuzyor} and we get
    $\widehat{V}_t-\widehat{V}_0=\int_0^t(I-\widehat{V}_u\widehat{V}_u^\top)\d W_u-\int_0^t\frac{n-1}{2}\widehat{V}_u\d u$
    so that $\widehat{V}$ is a Brownian motion on $\Sp^{n-1}$.
    
    To prove independence of $\widehat{V}$ and $R$ it is enough to prove independence of Brownian motions $W$ and $\theta$ which are driving the respective SDEs. Then $\widehat{V}$ and $R$ are strong solutions 
    to their corresponding SDEs which are driven by independent Brownian motions, so they are also independent. This holds since we note that for a strong solution $X$ of some SDE there exists a measurable map $\Phi,$ such that $X=\Phi(\widetilde{B}),$ where $\widetilde{B}$
    is Brownian motion driving the SDE c.f. \cite{chernybessel,yamadawatanabe}. Therefore we can find measurable maps $\Phi_1,\Phi_2$, such that $\widehat{V}=\Phi_1(W), R=\Phi_2(\theta)$ and independence does indeed follow from the independence of $\theta$ and $W$.
The    Markov property implies that $W$ depends on $\mathcal{G}_0=\F_s$ only
through $W_0=0,$ so $W$ is independent of $\F_s$. Hence $W$ is independent of
$(\theta_{t})_{t\in [0,s]}$. Therefore, it is enough to prove that $W$ is
independent of $(\theta_{t}-\theta_s)_{t\ge s}$. Define
$\eta_t:=\theta_{T_s(t)}-\theta_s =\int_s^{T_s(t)} V_u^\top\d {B}_u$ so that
$\eta$ is a $\mathcal{G}_t$-local martingale. Simple calculation shows that
$\langle W_i,\eta \rangle_t=0$ and $\langle \eta\rangle_t=T_s(t)-s$ with
inverse $S_s(t+s).$  We then use Knight's Theorem (also known as the multidimensional
Dambis-Dubins-Schwarz Theorem found in \cite[Chapter V, Theorem 1.9]{revuzyor}) to show
that $W$ and $(\eta_{S_s(t+s)})_{s\ge t}=(\theta_{s+t}-\theta_s)_{s \ge t}$ are
independent Brownian motions.
    
To address the last statement we need to consider the situation when
the solution is started from $0$.  The evolution of such a process is given by
$(R_t\phi_{S_s(t)},t\ge s)$ where $R$ is a square root of a solution to
SDE~\eqref{prav_u} and $\phi$ is an independent Brownian motion on the sphere
started at $\phi_0={X_s}/R_s.$ Due to \emph{rapid spinning} (i.e.
$\lim\limits_{s\downarrow0}S_s(t)=\infty$), the initial point
$\phi_0={X_s}/R_s$ will be forced to be uniformly distributed on the sphere.
This follows from the properties of the skew-product decomposition established in this proof,
Lemma~\ref{timechg} above and 
\cite[Lemma~3.12]{elldiff}.
\end{proof}

    \begin{proof}[Proof of Corollary~\ref{uniqlaw} ]
    For $n=1$, pathwise uniqueness holds so uniqueness in law follows trivially.
    Now let $n\ge2$. When $x_0\neq 0$, by Theorem~\ref{skew}, the
solution $X$ is a measurable functional of two indpendent processes $R$ and
$\widehat{V}$ with given laws. Hence the law of X is unique.

    Finally, we consider the case of $X_0=0.$ What follows is almost direct
application of the proof of Theorem~1.1 in \cite{elldiff}. For any $k\in \N$
and open set $U\subseteq\RR^k,$ define a measurable function $F_U\colon
(0,\infty)^k\to
[0,1],F_U(t_1,\ldots,t_k):=\P_\Psi[(\psi_{t_1},\ldots,\psi_{t_k})\in U],$ where
law $\P_\Psi[\cdot]$ is defined in~\cite[Lemma~3.7]{elldiff}.  Letting
$\mathcal{F}_\infty^R:=\sigma(R_u,u\in\RR_+)$, we apply  
Lemma~\ref{timechg} and Theorem~\ref{skew} above and \cite[Lemma~3.12]{elldiff} to get $\P\left[\left(X_{t_1}/R_{t_1},\ldots,
X_{t_k}/R_{t_k}\right)\in U\middle|
\mathcal{F}_\infty^R\right]=F_U(S_s(t_1),\ldots,S_s(t_k))\quad\text{a.s.}$ for
$0<s<t_1<\cdots<t_k.$ Hence $\P\left[\left(X_{t_1}/R_{t_1},\ldots,
X_{t_k}/R_{t_k}\right)\in U\right]=\E[F_U(S_s(t_1),\ldots,S_s(t_k))]$ and
therefore the finite-dimensional distributions of $({X_t}/{R_t},t>0)$ are determined uniquely
by $\P_\Psi[\cdot]$ and law of $R$. Moreover, 
law of $X$ is determined uniquely by the law of $(R,{X}/R),$ therefore uniquely
by $\P_\Psi[\cdot]$ and law of process $R$ solving SDE~\eqref{enacbaradij}
started at $0$.  \end{proof}

\subsection{Pathwise uniqueness  for SDE~\eqref{enacbax}}

		Let $X$ and $\widetilde{X}$ be solutions to the SDE~\eqref{enacbax} driven by the same Brownian motion $B$ and started at the same point $x_0\in \B^n.$ 
		Pathwise uniqueness clearly holds up to the hitting time of boundary, so by restarting argument it is enough to prove pathwise uniqueness for starting points $x_0$ on the boundary. 
		Furthermore, it is enough to prove that $X_t=\widetilde{X}_t$ for $t\le \tau_\varepsilon$ where $\tau_\varepsilon=\inf\set{t\ge0}{\abs{X_t}^2\wedge\absX{t}^2\le1-\varepsilon}$ for some $\varepsilon>0$ and without loss of generality we can clearly assume that $\varepsilon<\frac{1}{2}.$ 
		To prove equality of the processes we will apply method of
DeBlassie \cite{deblassie-uniq}. Namely, we wish to use Gronwall's lemma, but
due to non-Lipschitzness we cannot apply it directly to
$\E[|X_t-\widetilde{X}_t|^2].$ The idea of DeBlassie (and Swart before with
$p=\frac{1}{2}$) is to denote $Y:=1-\abs{X}^2, \widetilde{Y}:=1-\absX{}^2$ and
look at the process $W:=|X-\widetilde{X}|^2+(Y^p-\widetilde{Y}^p)^2$ for some
$p\in(\frac{1}{2},1).$   We then have
		\begin{equation*}\d Y_t=-2\gamma(\abs{X_t})
Y_t^{1/2}\sum_{i=1}^nX^i_t\d B^i_t- n\gamma^2(\abs{X_t}) Y_t\d
t+\left(2g(\abs{X_t})-(n-1)\gamma^2(\abs{X_t})\right)\abs{X_t}^2\d t.
\end{equation*}
A slight modification of 
\cite[Lemma~2.1]{deblassie-uniq} implies that for
$p>1+\frac{n-1}{2}-\frac{g(1)}{\gamma^2(1)}$ a formal application of It\^o's
formula for the mapping $x \mapsto x^p$ is justified.  Defining 
\begin{equation*}G(u):=g(u)-\frac{n-1}{2}\gamma^2(u)+(p-1)\gamma^2(u),
\end{equation*}we get 
\begin{equation*}\d Y^p_t= -2p\gamma(\abs{X_t}) Y_t^{p-1/2}\sum_{i=1}^nX^i_t\d
B^i_t+2pY^{p-1}_t\abs{X_t}^2\ind(Y>0)G(\abs{X_t})\d t -np\gamma^2(\abs{X_t})
Y^p_t\d t,\end{equation*}where $t\le\tau_\varepsilon$, $\abs{X_0}^2>1-\varepsilon,$ and 
$\varepsilon=\varepsilon(p)$ is chosen in such a way that
$p>1+\frac{n-1}{2}-\frac{g(u)}{\gamma^2(u)}$ for $u\in (1-\varepsilon(p),1]$.
The latter condition is necessary to keep second term on the right hand side negative to
allow use of Fatou's lemma. Furthermore,  $\int_0^t\ind(Y_s=0)\d s=0$ holds.
Note that essentially all necessary calculations and results are the same as in
\cite{deblassie-uniq} if we change their $g$ for $g-\frac{n-1}{2}\gamma^2$.
Subtracting the equations for $Y^p$ and $\widetilde{Y}^p$ we get \begin{align*}
			\d(Y^p-\widetilde{Y}^p)_t &= -2p\sum_{i=1}^n\left(\gamma(\abs{X_t})Y_t^{p-1/2}X^i_t-\gamma(\absX{t})\widetilde{Y}_t^{p-1/2}\widetilde{X}^i_t\right)\d B^i_t\\
			&+2p\left(Y^{p-1}_t\abs{X_t}^2\ind(Y>0)G(\abs{X_t})-\widetilde{Y}^{p-1}_t\absX{t}^2\ind(\widetilde{Y}>0)G(\absX{t})\right)\d t\\
			&-np\left(\gamma^2(\abs{X_t})Y_t^p-\gamma^2(\absX{t})\widetilde{Y}_t^p\right)\d t\\
			&=:dM_t+\I_1\d t + \I_2\d t
		\end{align*}
		and It\^o's formula yields 
				\begin{align*}
		\d(Y^p-\widetilde{Y}^p)^2_t &=2(Y_t^p-\widetilde{Y}_t^p)(dM_t+\I_1\d t + \I_2\d t)+4p^2\sum_{i=1}^n\left(\gamma(\abs{X_t})Y_t^{p-1/2}X^i_t-\gamma(\absX{t})\widetilde{Y}_t^{p-1/2}\widetilde{X}^i_t\right)\d t\\
		&=:d\widetilde{M}_t+2(Y_t^p-\widetilde{Y}_t^p)(\I_1\d t + \I_2\d t)+\I_3\d t.
		\end{align*}
		We can also compute 
		\begin{align*}
			\d |X_t-\widetilde{X}_t|^2&=2\sum_{i=1}^n(X_t^i-\widetilde{X}_t^i)\sum_{j=1}^n\left(\gamma(\abs{X_t})\sigma_{ij}(X_t)-\gamma(\absX{t})\sigma_{ij}(\widetilde{X}_t)\right)\d B_t^j\\
			&-2 \sum_{i=1}^n(X_t^i-\widetilde{X}_t^i)\left(g(\abs{X_t})X^i_j-g(\absX{t})\widetilde{X}^i_t\right)\d t\\
			&+\sum_{i,j=1}^n(\gamma(\abs{X_t})\sigma_{ij}(X_t)-\gamma(\absX{t})\sigma_{ij}(\widetilde{X}_t))^2\d t\\
			&=:\d N_t+\I_4 \d t + \I_5 \d t.
		\end{align*}
		The term $\I_5$ is the one which disallows direct use of Gronwall's lemma. It is singular in a sense that $\frac{\I_5}{W}$ can be arbitrarily large. Another singular term is $\I_3,$ but fortunately we also have negative singular term $2(Y_t^p-\widetilde{Y}_t^p)\I_1$, which will ensure that altogether we stay non-singular.
		We will bound all the terms $\I_k$  and since $\I_1,\I_2,\I_3$ and $\I_4$ are exactly the same\footnote{Our $G$ is defined slightly differently but it is still Lipschitz, so everything works.} as in \cite{deblassie-uniq} in Lemmas~3.1, 3.2, 3.4, and 3.5, respectively, we will not do the calculations but only summarize final results.
	 Let us introduce non-negative process $Z:=(Y^p-\widetilde{Y}^p)(\widetilde{Y}^{p-1}-Y^{p-1}).$ To make sense of $Z_t$ we implicitly multiply everything  by $\ind(Y_t>0,\widetilde{Y}_t>0)$. We will use this convention until the end of the proof. Then we have 
		\begin{align*}
			(Y_t^p-\widetilde{Y}_t^p)\I_1&\le-2pZ_t\abs{X_t}^2G(\abs{X_t})+C_1\varepsilon Z_t,\\
			\abs{\I_2}&\le C_2\left(|Y^p_t-\widetilde{Y}^p_t|+|X_t-\widetilde{X}_t|\right),\\
			\I_3&\le \frac{p(2p-1)^2}{1-p}\gamma^2(\abs{X_t})\abs{X_t}^2Z_t+C_3 |X_t-\widetilde{X}_t|^2+ C_3\varepsilon Z_t,\\
			\I_4&\le C_4|X_t-\widetilde{X}_t|^2,
		\end{align*}
		where constants $C_1,C_2,C_3,$ and $C_4$ are independent of $\varepsilon.$  Bound for $\I_5$ has to be done differently due to non-diagonal nature of our SDE.  By  straightforward computation e.g. by computing Frobenius norm of the matrix $\gamma(\abs{X_t})\sigma(X_t)-\gamma(\absX{t})\sigma(\widetilde{X}_t),$ we see that 
		\begin{align*}
			\I_5&=\gamma^2(\abs{X_t})\left(1-\sqrt{1-\abs{X_t}^2}\right)^2+ \gamma^2(\absX{t})\left(1-\sqrt{1-\absX{t}^2}\right)^2\\
			&-2\gamma(\abs{X_t})\gamma(\absX{t})\left(1-\sqrt{1-\abs{X_t}^2}\right)\left(1-\sqrt{1-\absX{t}^2}\right)\frac{\left(X_t\cdot\widetilde{X}_t\right)^2}{\abs{X_t}^2\absX{t}^2}\\
			&=\left(\gamma(\abs{X_t})Y_t^{1/2}-\gamma(\absX{t})\widetilde{Y}_t^{1/2}+\gamma(\absX{t}) -\gamma(\abs{X_t})\right)^2\\
			&+2\gamma(\abs{X_t})\gamma(\absX{t})\left(1-\sqrt{1-\abs{X_t}^2}\right)\left(1-\sqrt{1-\absX{t}^2}\right)\frac{\abs{X_t}^2\absX{t}^2-\left(X_t\cdot\widetilde{X}_t\right)^2}{\abs{X_t}^2\absX{t}^2}.
		\end{align*}
For the first term we can use Cauchy-Schwartz inequality to bound it from above by
$$2\left(\gamma(\abs{X_t})Y_t^{1/2}-\gamma(\absX{t})\widetilde{Y}_t^{1/2}\right)^2+2\left(\gamma(\absX{t})
-\gamma(\abs{X_t})\right)^2\leq
C\left(\varepsilon^{2-p}Z_t+|X_t-\widetilde{X}_t|^2\right),$$ where the inequality follows from  
the proof of \cite[Lemma~3.6]{deblassie-uniq} and the Lipschitz continuity of $\gamma$. 
		The second term can by our assumptions be bounded by $8(\sup\gamma)^2\left( \abs{X_t}^2\absX{t}^2-\left(X_t\cdot\widetilde{X}_t\right)^2\right).$ 
Using $2X_t\cdot\widetilde{X}_t=\abs{X_t}^2+\absX{t}^2-|X_t-\widetilde{X}_t|^2$ we get  
$\abs{X_t}^2\absX{t}^2-\left(X_t\cdot\widetilde{X}_t\right)^2
\le|X_t-\widetilde{X}_t|^2$. Hence
		we find $\I_5\le C_5(\varepsilon^{2-p}Z_t+|X_t-\widetilde{X}_t|^2)$ with $C_5$ independent of $\varepsilon.$ Using above facts we get 
		\begin{align*}
			2(Y_t^p-\widetilde{Y}_t^p)\I_1+\I_3+\I_5&\le-4pZ_t\abs{X_t}^2G(\abs{X_t})+C\varepsilon Z_t\\
			&+\frac{p(2p-1)^2}{1-p}\abs{X_t}^2Z_t+C |X_t-\widetilde{X}_t|^2+ C\varepsilon Z_t\\
			&+C(\varepsilon^{2-p}Z_t+|X_t-\widetilde{X}_t|^2)\\
			&=4pZ_t\gamma^2(\abs{X_t})\abs{X_t}^2\left(1-p+\frac{(2p-1)^2}{4(1-p)}+\frac{n-1}{2}-\frac{g(\abs{X_t})}{\gamma^2(\abs{X_t})}\right)\\
			&+C(2\varepsilon+\varepsilon^{2-p})Z_t+2C|X_t-\widetilde{X}_t|^2.
		\end{align*} Note that expression $1-p+\frac{(2p-1)^2}{4(1-p)}$ is minimized at $p=1-\frac{\sqrt{2}}{4}$ and the value is then $\sqrt{2}-1.$ Therefore we use initial assumption that $\frac{g(1)}{\gamma^2(1)}-\frac{n-1}{2}>\sqrt{2}-1$ to ensure the whole bracket is negative. Note also that such choice of $p$ implies $p=1-\frac{\sqrt{2}}{4}>1-(\sqrt{2}-1)>1+\frac{n-1}{2}-\frac{g(1)}{\gamma^2(1)}$ so all previous calculations are justifiable since we have necessary condition our use of Lemma~2.1 from \cite{deblassie-uniq}. Fixing $p=1-\frac{\sqrt{2}}{4}$ we then let $\varepsilon$ possibly be even smaller to ensure that $\frac{g(u)}{\gamma^2(u)}-\frac{n-1}{2}>\sqrt{2}-1+\delta$ holds on $(1-\varepsilon,1]$ for some small fixed $\delta>0$. The coefficient in front of $Z_t$ equals 
		$$4p\gamma(\abs{X_t})\abs{X_t}^2\left(\sqrt{2}-1+\frac{n-1}{2}-\frac{g(\abs{X_t})}{\gamma^2(\abs{X_t})}\right)+C(2\varepsilon+\varepsilon^{2-p})\le -4p(\inf\gamma^2)(1-\varepsilon)\delta+C(2\varepsilon+\varepsilon^{2-p}).$$
		Therefore, by letting $\varepsilon$ be small enough we ensure that this coefficient in front of non-negative $Z_t$ is negative and bound 
		$2(Y_t^p-\widetilde{Y}_t^p)\I_1+\I_3+\I_5\le C|X_t-\widetilde{X}_t|^2$ follows.
		Recall that 
		$$\d W_t= d\widetilde{M}_t+\d N_t+2(Y_t^p-\widetilde{Y}_t^p)\I_1\d t + 2(Y_t^p-\widetilde{Y}_t^p)\I_2\d t+\I_3\d t+\I_4 \d t + \I_5 \d t$$ and let $\widetilde{\tau}_m$ be a localizing sequence of stopping times for local martingale $\widetilde{M}+N$.
		 Then using above bounds and the fact that $\int_0^t\ind(Y_s=0 \ \mathrm{or} \ \widetilde{Y}_s=0)\d s=0$ yields
		 \begin{align*}
		 	\E[W_{t\wedge\tau_\varepsilon\wedge\widetilde{\tau}_m}]&=\E\left[\int_0^{t\wedge\tau_\varepsilon\wedge\widetilde{\tau}_m}\left(2(Y_s^p-\widetilde{Y}_s^p)(\I_1+\I_2)+\I_3+\I_4 + \I_5 \right)\ind(Y_s>0,\widetilde{Y}_s>0)\d s \right]\\
		 	&\le C\E\left[\int_0^{t\wedge\tau_\varepsilon\wedge\widetilde{\tau}_m}\left(|X_s-\widetilde{X}_s|^2 +2(Y^p_s-\widetilde{Y}^p_s)^2+2|Y^p_s-\widetilde{Y}^p_s||X_s-\widetilde{X}_s|\right)\d s  \right]\\
		 	&\le 3C\int_0^{t\wedge\tau_\varepsilon\wedge\widetilde{\tau}_m} \E\left[W_s\right]\d s  
		 \end{align*}
		 and Gronwall's lemma implies that $\E[W_t]=0$ and by non-negativity also $W_t=0$ for $t\le \tau_\varepsilon\wedge\widetilde{\tau}_m$. Letting $m\to \infty$ we get $W_t=0$ for $t<\tau_\varepsilon$. Therefore $X_t=\widetilde{X}_t$ for $t<\tau_\varepsilon$ and pathwise uniqueness 
in Theorem \ref{pathuniq} follows.

\bibliographystyle{amsalpha}
\bibliography{source}
     
\end{document}